\documentclass[a4paper,11pt]{amsart}

\usepackage{amssymb}  
\usepackage[colorlinks=true,backref=page]{hyperref} 

\usepackage[ansinew]{inputenc}

\usepackage{cite}

\renewcommand*{\backref}[1]{}
\renewcommand*{\backrefalt}[4]{\quad \tiny
 \ifcase #1 (\textbf{NOT CITED.})%
 \or    (Cited on page~#2.)%
  \else   (Cited on pages~#2.)%
  \fi}

\newtheorem{remark}{Remark}
\newtheorem{lemma}{Lemma}

\newtheorem{theorem}{Theorem}
\newtheorem{prop}{Proposition}
\newtheorem*{quest}{Question}

\newcommand{\Fix}{\mathrm{Fix}}
\newcommand{\R}{\mathbb{R}}
\newcommand{\Q}{\mathbb{Q}}
\newcommand{\Homeo}{\mathrm{Homeo}}

\title{Free orbits for minimal actions on the circle}
\author[J. Brum]{Joaqu\'in Brum}
\address{IMERL, Facultad de Ingenier\'ia, Universidad de la Rep\'ublica, Uruguay}
\email{joaquinbrum@fing.edu.uy}

\author[M. Mart\'inez]{Matilde Mart\'inez}
\address{IMERL, Facultad de Ingenier\'ia,Universidad de la Rep\'ublica, Uruguay} 
\email{matildem@fing.edu.uy}

\author[R. Potrie]{Rafael Potrie}
\address{CMAT, Facultad de Ciencias, Universidad de la Rep\'ublica, Uruguay}
\urladdr{www.cmat.edu.uy/$\sim$rpotrie}
\email{rpotrie@cmat.edu.uy}

\thanks{The authors were partially supported by CSIC grupo 618.}

\begin{document}

\maketitle

\begin{abstract}
We prove that if $\Gamma$ is a countable group without a subgroup isomorphic to $\mathbb{Z}^2$ that acts faithfully and minimally by orientation preserving homeomorphisms on the circle, then it has a free orbit. We give examples showing that this does not hold for actions by homeomorphisms of the line.  
\end{abstract}

\section{Introduction}

Foliations of codimension one and groups of homeomorphisms of the circle are closely related. 
  A particular but illuminating example of a foliation can be obtained via the suspension construction, by which an action of a surface group on the circle gives rise to a foliation on a circle bundle over a surface. In this example, fundamental groups of leaves correspond to stabilizers of points under the action, so that simply connected leaves translate into free orbits. When these foliations are minimal, either the generic leaf is simply connected or all leaves have a fundamental group which is not finitely generated (see \cite{ADMV}). With this motivation,  it is natural to ask if a minimal and faithful action of the fundamental group of a surface on the circle must have some free orbit.

It turns out that this is true in some greater generality and the purpose of this note is to prove the following result: 

\begin{theorem}
\label{maintheorem}
 Let $\Gamma$ be a countable group without a subgroup isomorphic to $\mathbb{Z}^2$.  If $\Gamma$ acts faithfully and minimally by orientation preserving homeomorphisms on the circle,
 then there exists a free orbit.
\end{theorem}

Recall that a \emph{free orbit} is the orbit of a point $x\in S^1$ such that for every $g \in \Gamma \setminus \{e\}$ one has that $gx \neq x$. 

Minimality of the action is necessary as it is shown by an example in subsection \ref{ss.nonminimal}. For actions on $\Homeo_+(\mathbb{R})$ the result is also non-valid, see subsection \ref{ss.countline}. 

It is natural to wonder whether a similar result will hold in higher dimensions. For example, one can ask: 

\begin{quest}
Is there a faithful and minimal action of the free group in two generators on a closed surface without free orbits?
\end{quest}

As a direct consequence of this results one deduces that if $f,g \in \mathrm{Homeo}_+(S^1)$ are homeomorphisms such that $f$ has a non-trivial interval of fixed points and $g$ is conjugate to an irrational rotation, then the group generated by $f$ and $g$ inside $\mathrm{Homeo}_+(S^1)$ is not free (and in particular contains a copy of $\mathbb{Z}^2$). It is illustrative to try to prove this consequence directly as it sheds light on the underlying ideas of our proof. 

\begin{remark}
One can also see that $\mathbb{Z}^2$ itself does not admit faithful minimal actions on the circle without free orbits. In fact, any group admitting such an action must be non-abelian, as we will see in Section \ref{s.improve}, where we give further conditions a group acting minimally and without free orbits must satisfy.
\end{remark}

For an excellent panoramic of the theory of group actions on the circle, see \cite{Ghys} or \cite{Navas}. 
We are grateful to Andr\'es Navas for his kind feedback on an early draft of this paper. The second author is also thankful to Fernando Alcalde, Fran\c{c}oise Dal'Bo and Alberto Verjovsky for many fruitful discussions about the topology of leaves of foliations by surfaces.

\section{Proof of Theorem \ref{maintheorem}}

We start with a simple remark which works for general countable groups. 

\begin{remark} \label{remarkinicial}
For each $g\in\Gamma\backslash\{e\}$, consider the set $\Fix(g)=\{x\in S^1:\ gx=x\}$ of its fixed points. The points with free orbit are exactly those in 
 $$\bigcap_{g\in\Gamma\backslash\{e\}}\Fix(g)^c.$$
If a countable group $\Gamma$ acts minimally on the circle and the action has no free orbit, then 
the following holds:
 \begin{enumerate}
 \item By Baire's Category Theorem there must exist $g\in \Gamma\backslash\{e\}$ such that $\Fix(g)$ has non-empty interior.
 \item Since the $\Gamma$-action is minimal on $S^1$, for every $x\in S^1$ there exists $h\in\Gamma\backslash\{e\}$ such that $x$ is an interior point of $\Fix(h)$.
\end{enumerate}
\end{remark}

Notice that the fact that $\Gamma$ is countable is crucial for the proof of this remark as it uses Baire category theorem. It is likely that arguments in the lines of the ones presented in \cite{BK} may help construct a non-countable group for which Theorem \ref{maintheorem} fails, however, we could not construct such and example and believe that this would exceed the purposes of this note. The main difficulty we encountered in approaching this problem can be summarized in the following question:

\begin{quest}
Is it possible to construct a map $\varphi: S^1 \to \mathrm{Homeo}_+(S^1)$ such that the group generated by the elements in the image of $\varphi$ is free? 
\end{quest}

We return to the proof of the Theorem. The following lemma will be the tool to obtain abelian subgroups.

\begin{lemma}\label{keylemma}
 Let $f$ and $g$ be two nontrivial orientation-preserving homeomorphisms of the circle. If $\Fix(f)\neq \Fix(g)$ and $\Fix(f)\cup \Fix(g)=S^1$, then
 the subgroup of $\mathrm{Homeo}_+(S^1)$ generated by $f$ and $g$ is isomorphic to $\mathbb{Z}^2$.
\end{lemma}

\begin{proof} Let $H\subset \Homeo_+(S^1)$ the subgroup generated by $f$ and $g$. We will begin by proving that $H$ is abelian. 

Notice that since $\Fix(f)\cup \Fix(g)=S^1$, we know that any point is either fixed by
$f$ or fixed by $g$.
Let $x\in S^1$. Without loss of generality, assume that $x\in \Fix(g)$. Therefore
$[f,g](x)=fgf^{-1}(x)$. If $x\in \Fix(f)$, then $x$ is fixed by both $f$ and $g$ and therefore by $[f,g]$. Otherwise, $f^{-1}(x)$ is not fixed 
by $f$ and is therefore fixed by $g$, so $[f,g](x)=x$. This implies that every point is fixed by $[f,g]$ and therefore $[f,g]=\mathrm{id}$ showing that $f$ and $g$ commute. 

Next, remark that since $\Fix(f)\neq \Fix(g)$ the group $H$ cannot be cyclic. Due to the classification of abelian groups, all
we have to see is that $H$ is torsion-free. Since the sets $\Fix(f)$ and $\Fix(g)$ are closed they cannot be disjoint, so any element of $H$ must have fixed points. This means that $H$ does not contain an element of finite order. 
\end{proof}

In order to prove Theorem \ref{maintheorem}, we will consider a countable group $\Gamma$ acting faithfully and minimally on $S^1$. Assuming that 
the action has no free orbit,  we will prove that $\Gamma$ contains a subgroup isomorphic to $\mathbb{Z}^2$. 

We will only use that $\Gamma$ is countable in order to use Remark \ref{remarkinicial} so that there is an element whose fixed point set has non-empty interior. Under this assumption, the result does not further use countability of $\Gamma$. 

\begin{proof}[Proof of Theorem \ref{maintheorem}] For every $x\in S^1$, consider the set 

$$A_x=\{I:\ I \mbox{ is an open interval in } S^1 \mbox{ and } x\in I\subset \Fix(g) \mbox{ for some } g\in \Gamma\backslash\{e\}\}.$$ 

Remark \ref{remarkinicial} guarantees that $A_x$ is non-empty for every $x\in S^1$. We fix an orientation in $S^1$. The orientation induces a total order on any interval $I$, and we can therefore write $I=(I_-,I_+)$. In particular, the interval $S^1\backslash\{x\}$ has an order, which allows us to consider suprema and infima of subsets of  $S^1\backslash\{x\}$.

Assume that for a given $x\in S^1$ the set $\mathcal{A}_x= \{I_+:\ I\in A_x\}$ is unbounded above in the total order of  $S^1\backslash\{x\}$. Consider $f\in\Gamma\backslash\{e\}$ such that $x$ is an interior point of $\Fix(f)$. Since $\mathcal{A}_x$ is unbounded, there exists $g\in \Gamma$ whose set of fixed points contains an interval $I$ such that $I\cup \Fix(f)=S^1$. In particular, $\Fix(f)\cup \Fix(g)=S^1$, and Lemma \ref{keylemma} implies
that $\Gamma$ contains a free abelian group of rank 2.

Otherwise, $\mathcal{A}_x$ must be bounded for all $x\in S^1$. In this case, we can define
$$h:S^1\to S^1,\ \ \ h(x)=\sup\mathcal{A}_x.$$

The map $h$ has the following properties which follow directly from its definition:
\begin{enumerate}
 \item it is monotonically increasing, (i.e.: any lift of $h$ to the line is a monotone map)
 \item it is equivariant, meaning that for every $g \in \Gamma$ and $x\in S^1$ one has 
  $gh(x)=h(gx)$. 
\end{enumerate}

Let us now show that $h$ is an homeomorphism. By equivariance, it follows that the image of $h$ is invariant by the $\Gamma$ action, therefore, by minimality it must be dense as otherwise $h$ would have a proper closed invariant subset. 
Now we check that $h$ has to be strictly monotonous. Consider $V=\bigcup_{x/int(h^{-1}(x))\neq\emptyset}int(h^{-1}(x))$. $V$ is an open, proper, $\Gamma$-invariant subset. The minimality of the action implies that $V$ is empty. Finally, since $h$ is strictly monotonous and has dense image, one obtains that $h \in \Homeo_+(S^1)$.

Now, we distinguish cases according to the rotation number of $h$ (\cite[Chapter 11]{KH}).

If $\rho(h)$ is irrational, then $h$ must be either a Denjoy counterexample or conjugated to a rational rotation. In the former case, $h$ has a countable union of intervals in its wandering set, which must be $\Gamma$-invariant since $h$ is equivariant. This is inconsistent with the minimality of the 
$\Gamma$-action. Therefore, $h$ is conjugated to an irrational rotation, and $\Gamma$ is isomorphic to a subgroup of the centralizer of $h$ in $\Homeo_+(S^1)$, but the centralizer of an irrational rotation does not have non-trivial elements with fixed points, which also gives a contradiction.

If $\rho(h)$ is rational, then $h$ has to be conjugate to a rigid rotation as otherwise the closed set of periodic points would be a proper closed invariant set for $\Gamma$ contradicting minimality. Assume then that $h^n=\mathrm{id}$, we will find $g$ and $g'$ in $\Gamma$ whose set of fixed points is different and whose union is $S^{1}$. For this, consider $x \in S^1$ and $g\in \Gamma$ such that $\Fix(g)$ contains $x$ in its interior. It follows that $\Fix(g)$ contains at least $n$ connected components each containing respectively $x, h(x), \ldots, h^{n-1}(x)$. Choose a point $y$ inside one of those components.
As $h$ is periodic, the orbit of $y$ by $h$ is alternated with the orbit of $x$.  By the definition of $h$ and its equivariance, we can find $g' \in \Gamma$ such that $\Fix(g) \cup \Fix(g') =S^1$ and both $g$ and $g'$ are not the identity. This allows one to apply Lemma \ref{keylemma} to conclude. 
\end{proof}

\section{Counterexamples}\label{s.nonminimal}

\subsection{A non-minimal action on $S^1$ without free orbits}\label{ss.nonminimal}

We construct here a faithful action of the fundamental group of a surface on the circle with no free orbits. The same can be obtained by adding a global fixed point to the example in the next subsection, but we present this example for the particular relevance of surface groups in actions on the circle. 

Let $\Gamma$ be the fundamental group of an oriented compact surface of genus greater than one. It does not contain any subgroup isomorphic to $\mathbb{Z}^2$. Surface groups are known to be $\omega$-residually free (see \cite{Ch-G}), which means that for any finite subset $X$ of $\Gamma$ there exists a homomorphism from $\Gamma$ to a free group whose restriction to $X$ is injective.

Consider a free subgroup $F$ of $\Homeo_+(\mathbb{R})$. Write $\Gamma=\cup_{n=0}^\infty X_n$ as an increasing union of finite subsets, and for each $n$ let $\varphi_n: \Gamma\to F$ be an homomorphism that sends $X_n$ injectively into $F$. Notice that, since $\Gamma$ is non-free, the Nielsen-Schreier theorem (see, for example, \cite[Section 2.2.4]{Stillwell}) implies that $\varphi_n$ must have a non-trivial kernel.
Take an increasing  sequence of points $(x_n)_{n=1}^\infty\bibliographystyle{}$ in $\mathbb{R}$ which does not accumulate in  $\mathbb{R}$. Taking $S^1$ to be $\mathbb{R}\cup\{\infty\}$ and setting $x_0=\infty$, the circle  is the union of the intervals $[x_n,x_{n+1}]$, for $n\geq 0$. We will identify each open interval $(x_n,x_{n+1})$ with the real line, so that $\varphi_n$ can be seen as a representation of $\Gamma$ in $\Homeo_+(x_n,x_{n+1})$.

We will define
$$\varphi:\Gamma\to\Homeo_+(S^1)$$ 
as follows:
\begin{itemize}
 \item[\textasteriskcentered] for any $g\in\Gamma$, $\varphi(g)$ fixes $\{x_n,\ n\geq 0\}$,
 \item[\textasteriskcentered] restricted to $(x_n,x_{n+1})$, $\varphi(g)$ coincides with $\varphi_n(g)\in\Homeo_+((x_n,x_{n+1}))$.
\end{itemize}

It is clear that $\varphi$ is a faithful representation, since each $\varphi_n$ is injective on $X_n$. We will see that the $\Gamma$-action defined by $\varphi$ has no free orbits. Let $x\in S^1$. If $x$ is not fixed by $\Gamma$, it belongs to $(x_n,x_{n+1})$ for some $n$, and it is therefore fixed by the non-trivial subgroup $\ker(\varphi_n)$.

\subsection{The main theorem does not hold in $\R$}\label{ss.countline}

The following example shows that Theorem \ref{maintheorem} is not true if we consider actions on the line.  
We will construct a faithful action of the free group $\mathbb{F}_2=\langle a,b\rangle$ on the line that is minimal and such that every point is stabilized by some non trivial element. 

We will start by defining three different $\mathbb{F}_2$ actions and later we will ``glue" them. 

Consider 
\begin{itemize}
\item $\phi_1\colon \mathbb{F}_2\to\R$ such that $\phi_1(a)(x)=x$ and $\phi_1(b)(x)=x+1$.
\item $\phi_2\colon \mathbb{F}_2\to\R$ such that $\phi_2(a)(x)=x+\alpha$ and $\phi_2(b)(x)=x+\beta$ for $\alpha$ and $\beta$ rationally independent over $\Q$. We also ask that $0<\alpha<1$ and $0<\beta<1$
\item $\phi_3\colon \mathbb{F}_2\to\R$ any action with a free orbit and without global fixed points.
\end{itemize}

Take $p>4$ such that $\phi_3(a)(p)>4+\alpha$ and $\phi_3(b)(p)>4+\beta$.
Define $f\in \Homeo_{+}(\R)$ satisfying $f(x)=\phi_1(a)(x)$ if $x<0$, $f(x)=\phi_2(a)(x)$ if $x\in[1,4]$ and $f(x)=\phi_3(a)(x)$ if $x>p$. 
Now we define $g\in \Homeo_{+}(\R)$ satisfying $g(x)=\phi_2(a)(x)$ if $x<0$, $g(x)=\phi_2(b)(x)$ if $x\in[1,4]$ and $g(x)=\phi_3(b)(x)$ if $x>p$. Finally define $f$ and $g$ over $[0,1]\cup[4,p]$ so that $Fix(f)\cap Fix(g)=\emptyset$.

Consider $\psi\colon\mathbb{F}_2\to\Homeo_{+}(\R)$ defines as $\psi(a)=f$ and $\psi(b)=g$. 
Since $\phi_3$ has a free orbit and $\phi_3$ has no global fixed point, for any $g\in\mathbb{F}_2-\{e\}$ there exists $x\in\R$ greater than $p$ such that $\phi_3(g)(x)\neq x$ and therefore $\psi(g)(x)\neq x$. This implies that $\psi$ is a faithful action. 

Now, the fact that $\psi$ has no global fixed points implies that given $x\in\R$ there exists $g\in\mathbb{F}_2$ so that $\psi(g)(x)<0$ and therefore $\psi(g^{-1}ag)(x)=x$ which proves that $\psi$ has no free orbit. 

It remains to check the minimality of $\psi$. Observe that given any $x\in[1,2]$ it is clear that the $\psi$ orbit of $x$ is dense on $[1,2]$. Now, since $\psi(a)([1,2])\cap[1,2]\neq\emptyset$ and $\psi(b)([1,2])\cap[1,2]\neq\emptyset$ we can deduce that $\mathbb{F}_2.[1,2]$, the union of the $\psi$ orbits of points in $[1,2]$, is a connected set. Also, since $\psi$ has no global fixed points $\mathbb{F}_2.[1,2]$ is unbounded in both directions and therefore $\mathbb{F}_2.[1,2]=\R$. Finally, any orbit accumulates on $[1,2]$ and therefore on $\R$ as claimed.

\begin{remark}
 Since any action on $\R$ can be seen as an action on $S^1=\R\cup\{\infty\}$ with a global fixed point, this is also an example of how Theorem \ref{maintheorem} can fail when the action is not minimal.
\end{remark}


\section{Further properties of minimal actions without free orbits}\label{s.improve}

\begin{remark}
 If $\Gamma$ is a non-cyclic group acting minimally and faithfully without free orbits con the circle, then it is non-abelian.
\end{remark}

To see this, consider an element $f\in\Gamma\backslash\{e\}$ such that $\Fix(f)$ is non-empty. If $\Gamma$ were abelian, the set $\Fix(f)$ would be invariant by all elements of $\Gamma$, so the action would not be minimal.

\begin{prop}
 If $\Gamma$ is a countable group acting minimally and faithfully without free orbits on the circle, then it contains a free group in two generators.
\end{prop}

\begin{proof} A result conjectured by Ghys and later proved by Margulis (see \cite{Margulis} or \cite{Navas}), states that any group of circle homeomorphisms either preserves a probability measure on $S^1$ or contains a free group in two generators. If $\Gamma$ acts without free orbits, it must be non-abelian. 

Suppose there is a $\Gamma$-invariant probability measure $\mu$. Since the action is minimal, it must have full support and no atoms. There is an homeomorphism sending $\mu$ to the Lebesgue measure. This means $\Gamma$ must be conjugated to a group of rotations, and therefore abelian, which gives a contradiction. 
 \end{proof}

\begin{prop}
If $\Gamma$ is a countable group acting minimally and faithfully without free orbits on the circle, then it contains free abelian groups of arbitrarily large rank. 
\end{prop}

\begin{proof} This follows by further inspection on the proof of our main theorem. We just sketch the proof. 

First, notice that $h$ is defined by contradiction and if it cannot be constructed it means that for every $x\in S^1$ there are elements for which there exist arbitrarily large intervals of fixed points containing $x$ (they contain the complement of arbitrarily small neighbourhoods of $x$). Notice that to obtain the conclusion, and in view of Lemma \ref{keylemma} it is enough to find, for a given $n>0$, elements $\gamma_1, \ldots, \gamma_n \in \Gamma$ so that their fixed point sets are different and pairwise cover $S^1$. This is possible under this assumption. 

Otherwise, one can construct $h$ and discuss similarly than in the proof of Theorem \ref{maintheorem}. We first recall that $\rho(h)$ must be rational. In this case one can argue as in the last paragraph to obtain such abelian groups. This completes the sketch of the proof. 
\end{proof}

\bibliographystyle{plain}
\bibliography{version2}{}

\end{document}